\documentclass{amsart}
\usepackage[utf8]{inputenc}
\usepackage[T1]{fontenc}
\usepackage{lmodern}
\usepackage[english]{babel}
\usepackage{amsmath}
\usepackage{amsfonts}
\usepackage{amssymb}
\usepackage[nobysame]{amsrefs}
\usepackage{graphicx}
\usepackage{mathrsfs}
\usepackage{tikz}
\usepackage[colorlinks=false,linktoc=page,pdftitle={On the inverse Klain map}]{hyperref}
\usepackage{paralist}

\theoremstyle{plain}
\newtheorem{theorem}{Theorem}[section]

\newtheorem{lemma}[theorem]{Lemma}
\newtheorem{corollary}[theorem]{Corollary}
\newtheorem*{theorem*}{Theorem}
\newtheorem*{corollary*}{Corollary}
\newtheorem*{QuestionA}{Question A}
\newtheorem*{QuestionB}{Question B}

\theoremstyle{definition}
\newtheorem{definition}[theorem]{Definition}

\newtheorem{remark}[theorem]{Remark}

\newcommand{\CC}{\mathbb{C}}
\newcommand{\DD}{\mathbb{D}}
\newcommand{\FF}{\mathbb{F}}

\newcommand{\RR}{\mathbb{R}}

\newcommand{\calF}{\mathcal{F}}
\newcommand{\calG}{\mathcal{G}}
\newcommand{\calH}{\mathcal{H}}
\newcommand{\calK}{\mathcal{K}}

\newcommand{\calM}{\mathcal{M}}
\newcommand{\calN}{\mathcal{N}}
\newcommand{\calP}{\mathcal{P}}

\newcommand{\rmR}{\mathrm{R}}
\newcommand{\rmr}{\mathrm{r}}

\DeclareMathOperator{\Val}{Val}
\DeclareMathOperator{\vol}{vol}
\DeclareMathOperator{\img}{im}

\DeclareMathOperator{\Klain}{Kl}

\DeclareMathOperator{\aff}{aff}

\DeclareMathOperator{\Grass}{Gr}

\usetikzlibrary{matrix,arrows}
\numberwithin{equation}{section}

\begin{document}
	\title[On the inverse Klain map]{On the inverse Klain map}

	\author{Lukas Parapatits}
	\address
	{
		\begin{flushleft}
			Vienna University of Technology \\
			Institute of Discrete Mathematics and Geometry \\
			Wiedner Hauptstraße 8--10/1046 \\
			1040 Wien, Austria \\
		\end{flushleft}
	}
	\email{lukas.parapatits@tuwien.ac.at}

	\author{Thomas Wannerer}
	\address
	{
		\begin{flushleft}	
			ETH Z\"urich\\
			Departement Mathematik \\
			HG F 28.3\\
			R\"amistrasse 101\\
			8092 Z\"urich, Switzerland \\
		\end{flushleft}
	}	
	
	\email{thomas.wannerer@math.ethz.ch}


	\begin{abstract}
		The continuity of the inverse Klain map is investigated and the class of centrally symmetric convex bodies
		at which every valuation depends continuously on its Klain function is characterized.
		Among several applications, it is shown that McMullen's decomposition is not possible in the class of
		translation-invariant, continuous, positive valuations.
		This implies that there exists no McMullen decomposition for translation-invariant, continuous Minkowski valuations,
		which solves a problem first posed by Schneider and Schuster.
	\end{abstract}
	
	\maketitle

	\section{Introduction}
		The Klain map was introduced by Klain in \cite{klain00} and has become an important tool in the theory
of even, translation-invariant, continuous valuations \cites{alesker01, alesker04, alesker_bernstein04, alesker_etal11},
which in turn led to spectacular advances in (Hermitian) integral geometry \cites{bernig09a,bernig09b, bernig_fu11,bernig_etal}.
The crucial fact that the Klain map is injective, which was first proved by Klain,
has been used to solve various problems related to even valuations and to provide different approaches to known results,
see e.g.\ \cites{klain00,schneider97, alesker00, alesker01}.
It is therefore of great interest to study the inverse Klain map.

Let $\Val_i^+$ denote the space of even, $i$-homogeneous, translation-invariant, continuous valuations
(see Section \ref{secdef} for precise definitions).
Given $\phi\in\Val^+_i$ and an $i$-dimensional, linear subspace $E$, we denote by $\phi_E$ the restriction of $\phi$ to $E$.
As a valuation on $E$, $\phi_E$ is even,  $i$-homogeneous, translation-invariant, continuous,
and hence, by a well-known theorem by Hadwiger \cite{Hadwiger1957}, proportional to the $i$-dimensional volume on $E$.
If we denote this proportionality factor by $\Klain_\phi(E)$, then
	$$\phi_E = \Klain_\phi(E) \vol_i.$$
The corresponding function  $\Klain_\phi \colon \Grass_i \to \CC$ on the Grassmannian of $i$-dimensional, linear subspaces
is called the Klain function of $\phi$.
The linear map $\Klain \colon \Val_i^+ \to C(\Grass_i), \ \phi \mapsto \Klain_\phi$ is called the Klain map.

In this article we investigate the dependence of $\phi$ on its Klain function,
i.e.\ the dependence of $\phi$ on its values on $i$-dimensional convex bodies.
Given a convex body $K$, it is important to know whether $\phi(K)$ depends \emph{continuously} on $\Klain_\phi$.
Let $\|\,\cdot\,\|$ denote the usual supremum-norm and $\calK(\RR^n)$ the set of convex bodies in $\RR^n$.

\begin{QuestionA}[Continuity of the inverse Klain map]
	Let $K\in\calK(\RR^n)$ be a convex body. Does there exist a constant $C\geq 0$ (depending on $K$) such that 
	\begin{equation}\label{eqKbounded}
		|\phi(K)|\leq C\|\Klain_\phi\|
	\end{equation}
	for all $\phi\in\Val_i^+$?
\end{QuestionA}

A different kind of dependence of $\phi$ on its values on $i$-dimensional convex bodies arises in the investigation of positive valuations,
i.e.\ valuations which satisfy $\phi(K)\geq 0$ for all $K\in\calK(\RR^n)$.
Here it is of interest to decide whether $\phi(L)\geq 0$ for $i$-dimensional convex bodies $L$ implies that $\phi(K)\geq 0$ for all convex bodies $K$.
For special classes of valuations such as constant coefficient valuations this is known to be true, see \cite{bernig_fu11}*{Corollary 2.10}.

\begin{QuestionB}[Monotonicity of the inverse Klain map]
	Let $K\in\calK(\RR^n)$ be a convex body. Is it true that 
	\begin{equation}\label{eqKpos}
		\Klain_\phi\geq0\ \Longrightarrow\ \phi(K)\geq 0
	\end{equation}
	for all $\phi\in\Val_i^+$?
\end{QuestionB}

As we will see, Question A and B are related in the sense that if $K\in\calK(\RR^n)$ satisfies \eqref{eqKpos} for every $\phi\in\Val_i^+$,
then also \eqref{eqKbounded} holds true for every $\phi\in\Val_i^+$. The answer to Question A is `yes' for generalized zonoids,
which are a dense subset of centrally symmetric convex bodies. In general however, the answer to Question~A or Question~B is negative.

\begin{theorem*} If $0<i<n-1$, then there exist centrally symmetric convex polytopes such that neither \eqref{eqKbounded} nor \eqref{eqKpos} hold for every $\phi\in\Val_i^+$.
\end{theorem*}
 
More precisely, we can characterize those centrally symmetric convex bodies for which the answer to Question A or Question B is positive. We need some notation to state the results.

Let $\calG(i)$, $0<i<n$, denote the class of those centrally symmetric convex bodies $K$ with the property
that there exists a signed Borel measure $\mu_K$ on $\Grass_i$ such that
\begin{equation}\label{eqintrep}
	\vol_i(K|E)=\int_{\Grass_i} \cos(E,F) \; d\mu_K(F)
\end{equation}
for each $E\in\Grass_i$.
Here $ \vol_i(K|E)$ denotes the $i$-dimensional volume of the orthogonal projection of $K$ on $E$
and $\cos(E,F)$ denotes the cosine of the angle between $E$ and $F$.
It is a well-known fact that if $K\in \calK(\RR^n)$ is a generalized zonoid, then $K\in \calG(i)$,
see e.g.\ \cite{weil79}*{Theorem 2.2} or \cite{schneider93}.
In particular, $\calG(i)$ lies dense in the space of centrally symmetric convex bodies.

The following theorem shows that the answer to Question~A is positive precisely for those convex bodies $K$ which lie in $\calG(i)$.

\begin{theorem*}
	Suppose that $K\in\calK(\RR^n)$ is centrally symmetric and that $0<i<n$.
	Then $K\in\calG(i)$ if and only if there exists a constant $C\geq0$ such that 
		$$|\phi(K)|\leq C\|\Klain_\phi\|$$
	for any $\phi\in\Val_i^+$.
\end{theorem*}

Extending a result of Goodey and Weil \cite{goodey_weil91}, we show in Theorem~\ref{proppoly}
that a centrally symmetric convex polytope $P$ is an element of $\calG(i)$ if and only if $P$ has centrally symmetric $(i+1)$-faces.
In particular, not every centrally symmetric convex body is an element of $\calG(i)$ for $i\neq n-1$. 

Let $\calK(i)$, $0<i<n$, denote the class of those centrally symmetric convex bodies $K$ with the property
that \eqref{eqintrep} holds with a positive Borel measure $\mu_K$.
These classes have been introduced by Weil in \cite{weil79} and were subsequently studied by Goodey and Weil in \cite{goodey_weil91}.
The class $\calK(1)$ coincides with the class of centrally symmetric zonoids and $\calK(n-1)=\calG(n-1)$ consists of all centrally symmetric convex bodies.
The following theorem shows in particular that the answer to Question~B is negative in general.

\begin{theorem*}
	Suppose that $K\in\calK(\RR^n)$ is centrally symmetric and that $0<i<n$. Then $K\in\calK(i)$ if and only if 
		$$0\leq\Klain_\phi \ \Longrightarrow\ 0\leq \phi(K),$$
	for every $\phi\in\Val_i^+$.
\end{theorem*}

The (dis-)continuity of the inverse Klain map underlines once more the importance of the notion of smooth valuation,
which was introduced by Alesker in \cite{alesker03}.
As a consequence of the characterization of the classes $\calG(i)$, we show in Corollary~\ref{corinvklain}
that the inverse of the Klain map $\Klain: \Val_i^+\rightarrow C(\Grass_i)$ is not continuous for $i\neq n-1$.
However, it was proved by Alesker and Bernstein that when restricted to $(\Val_i^+)^\infty$, the subspace of smooth valuations,
the Klain map  $\Klain: (\Val_i^+)^\infty\rightarrow C^\infty(\Grass_i)$ becomes an isomorphism of topological vector spaces onto its image
(see \cite{alesker_bernstein04}).
One consequence of the discontinuity of the inverse Klain map on continuous valuations is that Alesker's Fourier transform
$\FF:\Val^\infty\rightarrow \Val^\infty$ (see \cites{alesker03,alesker11}), which is an isomorphism of topological vector spaces,
cannot be continuously extended to a map $\Val\rightarrow \Val$.
As another consequence  of the characterization of the classes $\calG(i)$,
we obtain that not every even, translation-invariant, continuous valuations is angular.
More precisely, we show that there exist smooth valuations which are not angular.
This and related questions are discussed in greater detail in Section~\ref{secklain}.

One of the pillars of the theory of translation-invariant, continuous valuations is McMullen's decomposition theorem \cite{mcmullen77}.
It states that given a translation-invariant, continuous valuation $\phi$,
there exists for each $i=0,\ldots,n$ a unique translation-invariant, continuous valuation $\phi_i$ which is homogeneous of degree $i$ such that
\begin{equation}\label{eqmcmullen}
	\phi=\phi_0+\cdots +\phi_n.
\end{equation}
McMullen raised the question whether the same decomposition holds true in the class of translation-invariant, monotone valuations.
Here a valuation is called monotone if $K \subset L$ implies $\phi(K)\leq \phi(L)$.
Due to a theorem of McMullen \cite{mcmullen77}, every translation-invariant, monotone valuation is necessarily continuous.
Recently, Bernig and Fu \cite{bernig_fu11} have shown that
\emph{a translation-invariant, continuous valuation is monotone if and only if each of its homogeneous components is monotone},
thereby providing a positive answer to the question of McMullen.
However, the corresponding problem for positive valuations remained open.
Using that the inverse of the Klain map is not monotone (see Corollary \ref{corinvklain}),
we prove that an analog of McMullen's decomposition theorem does not hold in the class of positive, translation-invariant, continuous valuations.

\begin{theorem*}
	If $n\geq3$, then there exists a positive, even, translation-invariant, continuous valuation on $\calK(\RR^n)$
	such that not all of its homogeneous components are positive.
\end{theorem*}

First results on convex-body-valued valuations were obtained by Schneider \cite{schneider74} in the 1970s.
In recent years, the seminal work of Ludwig \cites{ludwig02,ludwig03,ludwig05,ludwig06}
has triggered more extensive investigations of such valuations
\cites{abardia12,abardia_bernig11,haberl_ludwig06,haberl08,haberl09,haberl11,schneider_schuster06,schuster07,schuster09,schuster_wannerer12,wannerer11,
haberl_parapatits,parapatits_1},
in particular in connection with the theory of affine isoperimetric inequalities
\cites{LutwakYangZhang2000_1, LutwakYangZhang2000_2, LutwakYangZhang2002, LutwakYangZhang2010}.
The existence of a decomposition \eqref{eqmcmullen} for convex-body-valued valuations would be very beneficial for this area,
since it would then be sufficient to study convex-body-valued valuations which are homogeneous of a certain degree.
A first step in this direction was taken by Schuster and the first named author \cite{parapatits_schuster},
who generalized the well known Steiner formula to translation-invariant, continuous Minkowski valuations $\Phi:\calK(\RR^n)\rightarrow \calK(\RR^n)$,
i.e.\ convex-body-valued valuations where the addition on the target space is the usual vector addition of sets.
The existence of a McMullen decomposition for Minkowski valuations would immediately imply this Steiner formula for Minkowski valuations.
However, from the result on positive valuations we immediately obtain that a McMullen decomposition for Minkowski valuations does not exist.
This solves a problem first posed by Schneider and Schuster \cite{schneider_schuster06} (see also \cites{parapatits_schuster, schuster09}).

\begin{corollary*}
	If $n\geq3$, then there exists an even, translation-invariant, continuous Minkowski valuation $\Phi:\calK(\RR^n)\rightarrow\calK(\RR^n)$
	which cannot be decomposed into a sum of homogeneous Minkowski valuations.
\end{corollary*}

	\section{Definitions and Background}\label{secdef}
		\subsection{Valuation theory}

We denote by $\Grass_i=\Grass_i(\RR^n)$ the Grassmann manifold of $i$-dimensional, linear subspaces of $\RR^n$
and by $\overline\Grass_i$ the Grassmannian of the corresponding affine subspaces.
We write $\calK(\RR^n)$ for the space of convex bodies of $\RR^n$, i.e. non-empty, convex, compact subsets of $\RR^n$,
and $\calP(\RR^n)$ for the set of polytopes in $\RR^n$.
We use $x\cdot y$ to denote the standard Euclidean inner product of $x,y\in\RR^n$ and $|x|=\sqrt{x\cdot x}$ for the Euclidean norm of $x\in\RR^n$.
The topology on $\calK(\RR^n)$ is induced by the Hausdorff metric on $\calK(\RR^n)$,
	$$d_H(K,L)=\inf\{r\geq 0: K\subset L+rB^n\ \text{and}\ L\subset K+rB^n\},\qquad K,L\in\calK(\RR^n),$$
where $B^n=\{x\in\RR^n: |x|\leq 1\}$ denotes the Euclidean unit ball of $\RR^n$.
We put 
	$$\omega_n=\vol_n(B^n)=\frac{\pi^\frac{n}{2}}{\Gamma(\frac{n}{2}+1)}$$
for the volume of the $n$-dimensional Euclidean unit ball of $\RR^n$ and set $S^{n-1}=\{x\in\RR^n:|x|=1\}$ for the Euclidean unit sphere.
We denote by
	$$h_K(x)=h(K,x)=\max\{x\cdot y: y\in K\},\qquad x\in\RR^n,$$
the support function of a convex body $K\in\calK(\RR^n)$.

\begin{definition}
	Let $(A,+)$ be an abelian semigroup.
	A map $\phi:\calK(\RR^n)\rightarrow A$ is called an $A$-valued \emph{valuation} if
	\begin{equation}\label{eq_defval}
		\phi(K\cup L)+\phi(K\cap L)=\phi(K)+\phi(L),
	\end{equation}
	whenever $K$, $L$, and $K\cup L\in\calK(\RR^n)$.
\end{definition}
In the case $(A,+)=(\CC,+)$ we speak of (scalar-valued) valuations.
In the case $(A,+)=(\calK(\RR^n),+)$, where `$+$' denotes the usual Minkowski sum of sets, i.e. \
	$$K+L=\{x+y: x\in K\ \text{and}\ y\in L\},$$
we speak of Minkowski valuations.
It follows from the basic properties of support functions (see e.g.\ Schneider \cite{schneider93})
that a map $\Phi: \calK(\RR^n)\rightarrow \calK(\RR^n)$ is a Minkowski valuation if and only if for every $x\in\RR^n$ the function
	$$K\mapsto h(\Phi K,x)$$
is a valuation.

A valuation $\phi$ is called continuous if it is continuous with respect to the topology induced by the Hausdorff metric.
We call $\phi$ translation-invariant if $\phi(K+x)=\phi(K)$ for every $x\in\RR^n$ and $K\in\calK(\RR^n)$
and we call $\phi$ homogeneous of degree $i$ if $\phi(tK)=t^i\phi(K)$ for any $t\geq 0$ and $K\in\calK(\RR^n)$.
The space of translation-invariant, continuous valuations is denoted by $\Val$ and the subspace of $i$-homogeneous valuations is denoted by $\Val_i$.
The following result is known as McMullen's decomposition theorem.

\begin{theorem}[McMullen \cite{mcmullen77}]\label{thmmcmullen}
	$$\Val=\bigoplus_{i=0}^n \Val_i.$$
\end{theorem}

The space $\Val_i$ can be further decomposed into even and odd valuations, $\Val_i=\Val_i^+\oplus\Val_i^-$,
where a valuation $\phi$ is called even (resp.\ odd) if $\phi(-K)=\phi(K)$ (resp.\ $\phi(-K)=-\phi(K)$) for every $K\in\calK(\RR^n)$.
It follows from McMullen's decomposition theorem that
	$$\|\phi\|=\sup\{|\phi(K)|: K\subset B^n\}$$
defines a Banach norm on $\Val$.
If we replace $B^n$ by a convex body $B$ with non-empty interior, we obtain an equivalent norm.

The canonical action of $GL(n)$, the group  of invertible, linear transformations of $\RR^n$, on $\Val$ is given by
	$$(g\cdot \phi)(K)=\phi(g^{-1} K).$$
A valuation $\phi$ is called smooth if $g\mapsto g\cdot \phi$ is a smooth map from $GL(n)$ to the Banach space $\Val$.
The subspace of smooth valuations is denoted by $\Val^\infty$
and carries a natural Fr\'{e}chet space topology via the identification of $\Val^\infty$ with the closed subspace of $C^\infty(GL(n), \Val)$
consisting of the smooth maps $g\mapsto g\cdot \phi$, $\phi\in\Val^\infty$.
Smooth valuations form a dense subspace of $\Val$.
For more information on smooth valuations see e.g.\ \cite{alesker03} or the survey article \cite{alesker07}.

We denote by $V(K_1,\ldots,K_n)$ the mixed volume of $K_1,\ldots,K_n\in\calK(\RR^n)$, which is normalized such that $V(K,\ldots,K)=\vol_n(K)$.
Given a partition $n=i_1+\cdots + i_m$, $i_j\geq1$, and convex bodies $K_1,\ldots,K_m\in\calK(\RR^n)$ we put
	$$V(K_1[i_1],\ldots, K_m[i_m])=V(\underbrace{K_1,\ldots,K_1}_{i_1\ \text{times}},\ldots,\underbrace{K_m,\ldots,K_m}_{i_m\ \text{times}}).$$
If $i_j=1$, we omit $[i_j]$. Fix $A_1,\ldots, A_{n-i}\in\calK(\RR^n)$.
Then
	$$K \mapsto V(K[i], A_1,\ldots,A_{n-i})$$
defines an element of $\Val_i$.
In particular, the $i$-th intrinsic volume is given by
	$$V_i(K)=\frac{\binom{n}{i}}{\omega_{n-i}}V(K[i],B^n[n-i]),$$
where the normalization is chosen such that for $i$-dimensional convex bodies $K$
the intrinsic volume $V_i(K)$ equals the ordinary $i$-dimensional volume of $K$.
The $i$-dimensional volume of the orthogonal projection of a convex body $K$ onto $E\in\Grass_i$ can be expressed as a mixed volume,
\begin{equation}\label{eqprojmixed}
	\vol_i(K|E)=\binom{n}{i} V(K[i],L_E[n-i]),
\end{equation}
where $L_E$ is any convex body contained in $E^\perp$, the orthogonal complement of $E$, with $\vol_{n-i}(L_E)=1$, see \cite{schneider93}*{p. 294}.

Given convex bodies $K_1,\ldots,K_{n-1}\in\calK(\RR^n)$
there exists a positive Borel measure $S(K_1,\ldots,K_{n-1},\;\cdot\;)$ on the Euclidean unit sphere
called the \emph{mixed area measure} of $K_1,\ldots,K_{n-1}$ which is uniquely determined by the property
	$$V(L,K_1,\ldots,K_{n-1})=\frac{1}{n} \int_{S^{n-1}} h(L,u)\; dS(K_1,\ldots,K_{n-1},u)$$
for every $L\in\calK(\RR^n)$.
For $K\in\calK(\RR^n)$ and $0\leq i<n$ the mixed area measure
	$$S_i(K,\;\cdot\;):=S(K[i], B^n[n-i-1],\;\cdot\;)$$
is called the $i$-th (Euclidean) area measure of $K$.
For a polytope $P\in \calK(\RR^n)$ the $i$-th area measure is given by the formula
\begin{equation}\label{eqsurfpoly}
	S_i(P,\omega)=\binom{n}{i}^{-1}\frac{n}{n-i} \sum_{F\in\calF_i(P)} \calH^{n-1-i}(N(F,P)\cap \omega)\vol_i(F).
\end{equation}
Here $N(F,P)$ denotes the normal cone of $P$ at $F$, $\calH^k$ the $k$-dimensional Hausdorff measure on $\RR^n$,
and $\calF_i(P)$ the set of $i$-faces of $P$.

McMullen conjectured that every continuous, translation-invariant valuation can be approximated by linear combinations of valuations of the form $K\mapsto \vol_n(K+A)$, $A\in\calK(\RR^n)$. The following theorem was proved by Alesker \cite{alesker01} and it provides---in a much stronger form---a positive solution to McMullen's conjecture. 

\begin{theorem}[Alesker \cite{alesker01}]\label{thmirred}
	For each $i$ the spaces $\Val_i^+$ and $\Val_i^-$ are irreducible $GL(n)$-representations,
	i.e. they do not have proper, invariant, closed subspaces. 
\end{theorem}

A valuation $\phi\in\Val$ is called simple if it vanishes on convex bodies with empty interior.
Klain \cite{klain95} proved that if $\phi$ is a translation-invariant, continuous, even, and simple valuation,
then $\phi$ must be a multiple of the $n$-dimensional volume, $\phi=c \vol_n$ for some constant $c\in\CC$.
Fix $0<i<n$.
Given $\phi\in \Val_i^+$ and $E\in \Grass_i$,  we denote by $\phi_E$ the restriction of $\phi$ to the subspace $E$.
As a valuation on $E$, $\phi_E$ is translation-invariant, continuous, even, and by Theorem \ref{thmmcmullen} simple.
Thus, $\phi_E$ is proportional to the $i$-dimensional volume on $E$ and we denote this proportionality factor by $\Klain_\phi(E)\in\CC$,
	$$\phi_E=\Klain_\phi(E) \vol_i, \qquad E\in\Grass_i.$$

Let $C(\Grass_i)$ denote the space of continuous functions on $\Grass_i$.
The map $\Klain: \Val_i^+\rightarrow C(\Grass_i)$, $\phi\mapsto \Klain_\phi$, is called Klain map
and it is easy to see that it is a continuous, linear operator.

\begin{theorem}[Klain \cite{klain00}]\label{thmklain}
	The Klain map $\Klain: \Val_i^+\rightarrow C(\Grass_i)$ is injective.
\end{theorem}

Finally, let us recall the following result of McMullen \cite{mcmullen83}.
Leaving the defining Equation \eqref{eq_defval} unchanged, we can also consider valuations defined on (subsets of) convex polytopes or polyhedral cones.
A valuation $\phi$ on $\calP(\RR^n)$ is called \emph{dilatation continuous} if $t\mapsto \phi(tP)$, $t\geq0$, is continuous for each $P\in\calP(\RR^n)$. A valuation
on the set of polyhedral cones with apex $0$ of dimension at most $i$ is called \emph{simple} if it vanishes on cones of dimension less than $i$. 

\begin{theorem}[McMullen \cite{mcmullen83}]\label{thmvalonpoly}
	A function $\phi:\calP(\RR^n)\rightarrow\CC$ is a translation-invariant and dilatation continuous valuation
	if and only if
		$$\phi(P)=\sum_{i=0}^n\sum_{F\in\calF_i(P)} \lambda_{n-i}(F,P)\vol_i(F),$$
	for every $P \in \calP(\RR^n)$,
	where $\lambda_{n-i}$ is a simple valuation on the set of polyhedral cones with apex $0$ of dimension at most $n-i$
	and $\lambda_{n-i}(F,P)=\lambda_{n-i}(N(F,P))$.
\end{theorem}

\subsection{Integral transformations on Grassmannians}

For integers $0<i,j<n$ and $F\in\Grass_j$, we denote by $\Grass_i^F$ the set of all linear subspaces $E\in \Grass_i$
for which either $E\subset F$, $E=F$, or $E\supset F$ depending on whether $i<j$, $i=j$, or $i>j$.

The \emph{Radon transform} is the continuous, linear operator $R_{ji}: C(\Grass_i)\rightarrow C(\Grass_j)$ defined by
	$$R_{ji}f(F)=\int_{\Grass_i^F} f(E)\; d\nu_F(E) , \quad F \in \Grass_j,$$
where $\nu_F$ denotes the unique rotation-invariant probability measure on $\Grass_i^F$.

It is well-known, see e.g.\ \cite{schneider_weil92}*{Satz 6.1.1}, that
\begin{equation}\label{euclid_eqhaarmeas}
	\int_{\Grass_j} \int_{\Grass_i^F} f(E,F)\; d\nu_F(E)\; dF= \int_{\Grass_i}\int_{\Grass_j^E} f(E,F)\;d\nu_E(F) dE,
\end{equation}
for each $f\in C(\Grass_i\times \Grass_j)$.
As a consequence, we obtain the relations
	$$R_{ki}= R_{kj}\circ R_{ji} \quad\text{and}\quad R_{ik}= R_{ij}\circ R_{jk}$$
for all integers $0< i\leq j\leq k<n$.
Since
	$$\perp\circ R_{n-i,n-j}=R_{ij}\circ\perp$$
each one of the above relations implies the other.
Let $L_2(\Grass_i)$ denote the Hilbert space of square-integrable functions on $\Grass_i$.
The Radon transform can be extended to a continuous, linear operator $R_{ji}: L_2(\Grass_i)\rightarrow L_2(\Grass_j)$.
An application of (\ref{euclid_eqhaarmeas}) shows that $R_{ij}$ is the adjoint of $R_{ji}$, that is
\begin{equation}\label{euclid_eqradadjoint}
	\left( R_{ji} f, g\right)=\left(f,  R_{ij} g\right)
\end{equation}
for all $f\in L_2(\Grass_i)$ and $g\in L_2(\Grass_j)$.
Here $( \, \cdot \, {,} \, \cdot \, )$ denotes the $L_2$ inner product on the spaces $L_2(\Grass_i)$ and  $L_2(\Grass_j)$, respectively.

We remark that the Radon transform is an intertwining operator,
i.e.\ the Radon transform commutes with the $SO(n)$-actions on the spaces $C(\Grass_i)$ (and hence also $L_2(\Grass_i)$),
maps smooth functions to smooth functions, and $R_{ji}: C^\infty(\Grass_i)\rightarrow C^\infty(\Grass_j)$ is continuous in the $C^\infty$-topology.
Using relation \eqref{euclid_eqradadjoint}, we can easily extend the domain of the Radon transform to measures or distributions.
For a signed Borel measure on $\Grass_i$ we define $R_{ji}\mu$ by
\begin{equation}\label{defradonmeasure}
	\int_{\Grass_j} g\; d(R_{ji}\mu)= \int_{\Grass_i} R_{ij}g\; d\mu,\qquad g\in C(\Grass_j).
\end{equation}
More information on the Radon transform may be found in \cites{helgason00,helgason11}.
Recently, Alesker \cite{alesker10} established the existence of a Radon transform of smooth valuations
which contains the classical Radon transform of smooth functions and the Radon transform of constructible functions as special cases.

For $E, F\in\Grass_i$ the \emph{cosine of the angle between $E$ and $F$} is the number
\begin{equation}\label{euclid_eqcosdef}
	\cos(E,F) :=\frac{\vol_i(A|E)}{\vol_i(A)},
\end{equation}
where $A\subset F$ is a convex body of non-zero $i$-dimensional volume.
It is not difficult to show that
\begin{equation}\label{euclid_eqcossym}
	\cos(E,F) = \cos(F,E), \qquad \ \cos(E^\perp,F^\perp) = \cos(E,F),
\end{equation}
and
\begin{equation}\label{euclid_eqcos1}
	 \cos(E,F) \leq 1.
\end{equation}

The \emph{cosine transform} is the continuous, linear operator $C_i: C(\Grass_i)\rightarrow C(\Grass_i)$, defined by
	$$C_if(E)=\int_{\Grass_i} \cos(E,F) f(F)\; dF, \quad E \in \Grass_i.$$

Using that $\|f\|_1\leq \|f\|_2$ on probability measure spaces and \eqref{euclid_eqcos1}, we obtain
	$$\|C_if\|^2_2=\int_{\Grass_i} \left(\int_{\Grass_i} \cos(E,F) f(F)\; dF\right)^2dE\leq \|f\|_2^2 $$
for all $f\in C(\Grass_i)$.
Hence, also the cosine transform can be extended to a continuous operator on $L_2(\Grass_i)$.
It follows from (\ref{euclid_eqcossym}) that the cosine transform is a self-adjoint operator, i.e.\
	$$(C_if,g)=(f,C_ig)$$
for all $f,g\in L_2(\Grass_i)$ and that
\begin{equation}\label{euclid_eqcosperp}
	\perp \circ C_i=C_{n-i}\circ \perp.
\end{equation}

Like the Radon transform, the cosine transform is an $SO(n)$-intertwining operator.
As in the case of the Radon transform, this implies that $C_i$ maps smooth functions to smooth functions
and that $C_i: C^\infty(\Grass_i)\rightarrow C^\infty(\Grass_i)$ is continuous in the $C^\infty$-topology.
For a signed Borel measure on $\Grass_i$ we define $C_i\mu\in C(\Grass_i)$, the cosine transform  of $\mu$, by
	$$C_i\mu(E)=\int_{\Grass_i}  \cos(E,F) \; d\mu(F), \quad E \in \Grass_i.$$

For more information on the cosine transform we refer the reader to \cite{goodey_zhang98} and \cite{alesker_bernstein04}.
The cosine transform plays an important role in the theory of translation-invariant, continuous valuations (see \cite{alesker03}).
This connection works both ways, as the image of the cosine transform was determined in \cite{alesker_bernstein04} using results from valuation theory.

	\section{\texorpdfstring{Polytopal members of $\calG(i)$}{Polytopal members of G(i)}}
		The aim of this section is to give a characterization of the polytopal members of $\calG(i)$,
slightly generalizing results obtained by Goodey and Weil \cite{goodey_weil91} for $\calK(i)$.
For the convenience of the reader we have decided to include a complete proof
instead of referring to the various arguments and results scattered in \cite{goodey_weil91} and pointing out modifications.

If $K\in\calK(\RR^n)$ is centrally symmetric, then $\vol_1(K|\RR u)=2h(K,u)$, $u\in S^{n-1}$.
Hence, $K\in\calK(1)$ (or $K\in\calG(1)$) if and only if $K$ is a (generalized) zonoid, see e.g.\ \cite{schneider93}*{Section 3.5}.
Moreover, since
	$$\vol_{n-1}(K|u^\perp)=\frac{1}{2}\int_{S^{n-1}} |u\cdot v|\; dS_{n-1}(K,u), \qquad u\in S^{n-1},$$
we obtain that $\calK(n-1)=\calG(n-1)$ consists of all centrally symmetric convex bodies.
From \cite{schneider93}*{Theorem 5.3.1} one can deduce that $\calK(1)\subset\calK(i)$ and $\calG(1)\subset\calG(i) $, see e.g.\ \cite{weil79}*{Theorem 2.2}.
Thus, we have the inclusions
	$$\calK(1)\subset\calK(i)\subset\calK(n-1)\quad \text{and}\quad \calG(1)\subset\calG(i)\subset\calG(n-1).$$

For $E\in \Grass_i$ we denote by $V(K_1,\ldots,K_{i}:E)$ the mixed volume of the
orthogonal projections of $K_1, \ldots, K_{i}$ on $E$ computed in $E$.
\begin{lemma}[Schneider \cite{schneider97}]\label{lemproj}
	If $K\in\calG(i)$, then
	\begin{equation}\label{eqproj}
	V(K[i],L_1,\ldots,L_{n-i})=\binom{n}{i}^{-1}\int_{\Grass_i} V(L_1,\ldots,L_{n-i}:E^\perp)\; d\mu_K(E)
	\end{equation}
for any convex bodies $L_1,\ldots,L_{n-i}\in \calK(\RR^n)$.
\end{lemma}

\begin{remark}
	Goodey and Weil \cite{goodey_weil91} proved Lemma~\ref{lemproj} under the additional assumption
	that $L_1,\ldots,L_{n-i}$ are \emph{centrally symmetric} and raised the question whether (\ref{eqproj}) holds true in general.
	The proof given below using the Klain map is due to Schneider \cite{schneider97}.
	For completeness and since it is a nice application of the injectivity of the Klain map, we have included it here.
\end{remark}

\begin{proof}[Proof of Lemma~\ref{lemproj}]
	We define  valuations $\phi,\psi\in\Val_{n-i}^+$ by
		$$\phi(L)=V(K[i],L[n-i]) \quad\text{and}\quad \psi(L)=\binom{n}{i}^{-1}\int_{\Grass_i} \vol_{n-i}(L|E^\perp)\; d\mu_K(E).$$
	From the definition of the Klain map, \eqref{eqprojmixed}, \eqref{eqintrep}, and \eqref{euclid_eqcossym} we deduce that
		$$\Klain_\phi(F)= \binom{n}{i}^{-1} \vol_i(K|F^\perp)=  \binom{n}{i}^{-1} \int_{\Grass_i} \cos(F,E^\perp) \; d\mu_K(E).$$
	On the other hand, from the definition of the cosine of the angle between two subspaces (\ref{euclid_eqcosdef}) we obtain
		$$\Klain_\psi(F)= \binom{n}{i}^{-1} \int_{\Grass_i} \cos(F,E^\perp) \; d\mu_K(E).$$
	Thus, $\Klain_\phi=\Klain_\psi$.
	The  injectivity of the Klain map (Theorem~\ref{thmklain}) therefore yields $\phi=\psi$.
	We conclude that (\ref{eqproj}) holds if $L=L_1=\ldots= L_{n-i}\in\calK(\RR^n)$.
	To deduce the general case from this, just put $L=\lambda_1L_1+\cdots+\lambda_{n-i} L_{n-i}$
	and expand both sides of (\ref{eqproj}) into polynomials in $\lambda_1,\ldots,\lambda_{n-i} \geq 0$.
\end{proof}

If $K\in\calK(\RR^n)$ is origin symmetric, then $h_K$ is an even function on the sphere and hence can be viewed as a function on $\Grass_1$.
Similarly, if $\mu$ is a symmetric measure on $S^{n-1}$ we may view $\mu$ as a measure on $\Grass_1$ and vice versa.

\begin{lemma}[Goodey and Weil \cite{goodey_weil91}]\label{lemradon}
	If $K\in\calK(\RR^n)$ is origin symmetric and $1\leq i\leq n-1$, then
		$$V(K,B^n[i-1]: E)=\omega_{i} R_{i,1}(h_K)(E), \qquad E\in\Grass_{i}.$$
\end{lemma}

\begin{proof}
	Let $B_{E^\perp}$ be the $(n-i)$-dimensional ball in $E^\perp$ centered at the origin with $\vol_{n-i}(B_{E^\perp}) = 1$.
	By \eqref{eqprojmixed} we have
		$$ \vol_i(K|E) = \binom{n}{i} V(K[i],B_{E^\perp}[n-i]). $$
	Similar to the proof of Lemma \ref{lemproj} we get
	\begin{equation}\label{eqmixedvolproj}
		V(K_1,\ldots,K_i:E) = \binom{n}{i} V(K_1,\ldots,K_i,B_{E^\perp}[n-i])
	\end{equation}
	for all $ K_1,\ldots,K_i \in \calK(\RR^n)$.
	Using the definition of the $(n-i)$-th surface area measure, we get for $K \in \calK(\RR^n)$
	\begin{align*}
		V(K,B[i-1]:E)
		&= \binom{n}{i} V(K,B^n[i-1],B_{E^\perp}[n-i]) \\
		&= \binom{n}{i} \frac 1 n \int_{S^{n-1}} h(K,u) \; dS_{n-i}(B_{E^\perp},u). \\
	\end{align*}
	Since the measure $S_{n-i}(B_{E^\perp},\;\cdot\;)$ is uniformly distributed on $S^{n-1} \cap E$
	with total mass $\binom{n}{i}^{-1} n \omega_i$ (cf.\ Equation \eqref{eqsurfpoly}),
	we get the desired equation.
\end{proof}

The following theorem gives the desired characterization of the polytopal members of $\calG(i)$.

\begin{theorem}\label{proppoly}
	Let $P\in \calK(\RR^n)$ be a centrally symmetric polytope.
	Then  $P\in\calG(i)$ if and only if $P$ has centrally symmetric $(i+1)$-faces.
	In particular, $\calG(i)\neq \calG(n-1)$ if $i\neq n-1$.
\end{theorem}

\begin{proof}
	Suppose $P\in\calG(i)$ is a polytope.
	From (\ref{eqproj}) we obtain
		$$V(P[i],L,B^n[n-i-1])=\binom{n}{i}^{-1}\int_{\Grass_i} V(L,B^n[n-i-1]:E^\perp)\; d\mu_P(E)$$
	whenever  $L\in\calK(\RR^n)$ is an origin symmetric convex body.
	Using Lemma~\ref{lemradon} and \eqref{defradonmeasure}, we arrive at
	\begin{align*}
		V(P[i],L,B^n[n-i-1])&=\omega_{n-i}\binom{n}{i}^{-1}\int_{\Grass_i} R_{n-i,1}h_L(E^\perp)\; d\mu_P(E)\\
		&=\omega_{n-i}\binom{n}{i}^{-1}\int_{\Grass_{n-i}} R_{n-i,1}h_L(E)\; d\mu_P^\perp(E)\\
		&=\omega_{n-i}\binom{n}{i}^{-1}\int_{\Grass_1} h_L\; d(R_{1,n-i}\mu_P^\perp).
	\end{align*}
	But on the other hand
		$$V(P[i],L,B^n[n-i-1])= \frac{1}{n}\int_{S^{n-1}} h_L(u)\; dS_i(P,u).$$
	Since this holds for all $L$, we conclude
	\begin{equation}\label{eqarearad}
		S_i(P,\;\cdot\;)=n\omega_{n-i}\binom{n}{i}^{-1} R_{1,n-i}\mu_P^\perp.
	\end{equation}

	Let $\mu$ be a signed Borel measure on $\Grass_{n-i}$.
	We claim that the measure $R_{1,n-i}\mu$ is evenly distributed on lines contained in $(n-i)$-dimensional subspaces.
	Indeed, fix a subspace $E\in\Grass_{n-i}$ and let $A\subset \Grass_1$ be a Borel set such that every element of $A$ is a subset of $E$.
	We compute
		$$R_{1,n-i}\mu(A) = \int_{\Grass_{n-i}} R_{n-i,1}\chi_A\; d\mu = \nu_E(A)\mu(\{E\}),$$
	where $\nu_E$ is the Haar measure on $\Grass_1^E$, the set of lines contained in $E$.
	Since the support of $S_i(P,\;\cdot\;)$ is contained in a finite union of $(n-i)$-dimensional subspaces, see \eqref{eqsurfpoly},
	and by \eqref{eqarearad} the measure $S_i(P,\;\cdot\;)$ is the Radon transform of some signed measure on $\Grass_{n-i}$,
	we obtain that there are positive constants $\alpha_1,\ldots,\alpha_m$ and subspaces $E_1,\ldots,E_m\in\Grass_{n-i}$ such that
		$$S_i(P,A)=\sum_{j=1}^m \alpha_j \calH^{n-i-1}(A\cap E_j)$$
	for every Borel set $A\subset S^{n-1}$.
	From this together with \eqref{eqsurfpoly}, we deduce
	\begin{equation}\label{equnioncones}
		\bigcup_{j=1}^m E_j= \bigcup_{F\in \calF_i(P)} N(F,P)
	\end{equation}
	and that parallel $i$-faces have the same $i$-dimensional volume.

	Let $G$ be an $(i+1)$-face of $P$ and $F$ and $i$-face of $G$.
	Recall that $N(G,P)$ is a facet of $N(F,P)$.
	By \eqref{equnioncones} there must be an $i$-face $F'$ of $P$ parallel to $F$ such that $N(G,P)$ is also a facet of $N(F',P)$.
	Therefore $F'$ must be contained in $G$.
	Hence the $i$-faces of $G$ appear in parallel pairs of the same volume. It is well-known that two solutions of the Minkowski problem differ only by a translation. Considering the Minkowski problem in the affine subspace spanned by $G$, we deduce that $-G$ must be a translate of $G$. Hence $G$ is centrally symmetric.

	Conversely, let $0<i<n-1$ and suppose that $P$ is a polytope with centrally symmetric $(i+1)$-faces.
	Let $F$ be an $i$-face of $P$ and suppose $G$ is any $(i+1)$-face containing $F$.
	Since $G$ is centrally symmetric, the reflection in the center of $G$ carries $F$ to a translate of $-F$.
	Any $(i+1)$-face containing this translate must contain for the same reason a translate of $F$, and so on.
	Let $P'$ denote the projection of $P$ onto the subspace orthogonal to $F$.
	The $i$-face $F$ corresponds to a vertex $F'$ of $P'$ and the $(i+1)$-faces containing $F$ correspond to edges of $P'$ containing $F'$.
	Note that any vertex of $P'$ can be connected by an edge path to the vertex $F'$.
	Together with the above we see that vertices of $P'$ are projections of $i$-faces of $P$ parallel to $F$.
	This implies that the union of the normal cones of $i$-faces parallel to $F$ is an $(n-i)$-dimensional subspace.

	Fix $L\in\Grass_i$.
	Applying Theorem~\ref{thmvalonpoly} to the valuation $\phi(P)=\vol_i(P|L)$,
	we obtain that there exists a simple valuation $\lambda_{n-i}$ on the set of polyhedral cones with apex $0$ whose dimension is at most $n-i$,
	such that
	\begin{equation}\label{eqvalonpoly}
		\vol_i(P|L)=\sum_{F\in\calF_i(P)} \lambda_{n-i}(F,P)\vol_i(F)
	\end{equation}
	for every $P\in\calP(\RR^n)$.
	If $P$ is a polytope with centrally symmetric $(i+1)$-faces, using that $\lambda_{n-i}$ is simple,
	we conclude from the above that there exist subspaces $E_1,\ldots, E_m\in\Grass_{n-i}$ and positive numbers $\alpha_1,\ldots,\alpha_m$ such that
		$$\vol_i(P|L)=\sum_{j=1}^m \alpha_j\lambda_{n-i}(E_j).$$
	Plugging $i$-dimensional polytopes into \eqref{eqvalonpoly}, we find that
		$$\lambda_{n-i}(E)= \cos(E^\perp,L), \qquad E\in\Grass_{n-i}.$$
	Thus,
		$$\vol_i(P|L)=\sum_{j=1}^m \alpha_j \cos(E_j^\perp,L)$$
	and hence $P\in \calK(i)\subset \calG(i)$.
	This concludes the proof of the theorem.
\end{proof}

It is well-known that every generalized zonoid which is a polytope is in fact a zonoid, see e.g. \cite{schneider93}*{Corollary 3.3.6}.
This can be expressed as
	$$\calK(1)\cap \calP^n=\calG(1)\cap\calP^n.$$
The above proof also yields a more general version of this.

\begin{corollary}
	$$\calK(i)\cap \calP^n = \calG(i)\cap \calP^n, \qquad 0<i<n.$$
\end{corollary}

We put $\calP(i):=\calK(i)\cap \calP^n= \calG(i)\cap \calP^n$.

\begin{corollary}
	Let $n\geq 4$.
	Then
		$$\calP(1)=\cdots = \calP(n-3)\subsetneq \calP(n-2)\subsetneq \calP(n-1).$$
\end{corollary}

\begin{proof}
	McMullen \cite{mcmullen70} proved that if $P\in\calP(\RR^n)$ has centrally symmetric $k$-faces for some $k$, $2\leq k\leq n-2$,
	then the faces of $P$ of any dimension are centrally symmetric.
	In the same article McMullen gives an example of a polytope with centrally symmetric $(n-1)$-faces
	such that not every $(n-2)$-face is centrally symmetric.
	Furthermore, it was shown by Shephard \cite{shephard67} (see also McMullen \cite{mcmullen76})
	that if $P$ has centrally symmetric $k$-faces for some $k$, $2\leq k\leq n-1$, then $P$ has centrally symmetric $(k+1)$-faces.
\end{proof}

\begin{remark}
	Recall that $\calG(1)\subset \calG(i)\subset \calG(n-1)$ for $1\leq i\leq n-1$.
	It is an open problem whether the inclusion
		$$\calG(i) \subset\calG(j)$$
	holds true for general $0< i\leq j<n$.
\end{remark}

	\section{The Klain map}\label{secklain}
		In this section we show that the question whether a centrally symmetric convex body $K\in \calK(\RR^n)$ belongs to $\calG(i)$ or $\calK(i)$
can be decided using valuations from $\Val_i^+$.

\begin{theorem}\label{thmGi2}
	Let $K\in\calK(\RR^n)$ be centrally symmetric and $0<i<n$.
	Then $K\in\calG(i)$ if and only if there exists a constant $C\geq0$ such that
	\begin{equation}\label{eqbounded}
		|\phi(K)|\leq C\|\Klain_\phi\|
	\end{equation}
	for any $\phi\in\Val_i^+$.
\end{theorem}

\begin{proof}
	Fix $K\in\calG(i)$ and let $\phi\in\Val_i^+$ be the valuation defined by
		$$\phi(M)=V(M[i],L_1,\ldots,L_{n-i}), \qquad M\in\calK(\RR^n),$$
	where $L_1,\ldots,L_{n-i}$ are centrally symmetric convex bodies.
	It follows from \eqref{eqmixedvolproj} that $\Klain_\phi(E) = \binom{n}{i}^{-1} V(L_1,\ldots,L_{n-i}:\; E^\perp)$.
	Hence we can rewrite formula (\ref{eqproj}) as
	\begin{equation}\label{eqklainmu}
		\phi(K)=\int_{\Grass_{i}} \Klain_\phi(E)\; d\mu_K(E).
	\end{equation}
	By Alesker's irreducibility theorem, linear combinations of valuations of the form $\phi(M)=V(M[i],L_1,\ldots,L_{n-i})$
	lie dense in $\Val_i^+$, see Theorem~\ref{thmirred}.
	By the continuity of the Klain map, we conclude that (\ref{eqklainmu}) holds for any $\phi\in\Val_i^+$.
	Thus, there exists a constant $C\geq 0$ such that (\ref{eqbounded}) holds for any $\phi\in\Val_i^+$.

	Conversely, if (\ref{eqbounded}) holds, then the Hahn-Banach theorem and the Riesz representation theorem
	imply the existence of a signed Borel measure $\mu_K$ on $\Grass_i$ such that (\ref{eqklainmu}) holds for any $\phi\in\Val_i^+$.
	Fix $F\in \Grass_i$ and put $\phi(K)=\vol_i(K|F)$.
	Clearly, $\phi\in\Val_i^+$ and $\Klain_\phi(E)= \cos(E,F)$.
	Thus, (\ref{eqklainmu}) implies (\ref{eqintrep}), that is, $K\in\calG(i)$.
\end{proof}

We define a partial order on $\Val_i$ by $\phi\leq \psi$
if and only if $\phi$ and $\psi$ are real-valued and $\phi(K)\leq \psi(K)$ for every $K\in\calK(\RR^n)$.
Similarly, we define a partial order on $C(\Grass_i)$.
Since the Klain map is injective (Theorem~\ref{thmklain}), we can consider its inverse $\Klain^{-1}: \img(\Klain) \rightarrow \Val_i^+$.

\begin{corollary}\label{corinvklain}
	If $i\neq n-1$, then the inverse of the Klain map $\Klain^{-1}: \img(\Klain) \rightarrow \Val_i^+$ is not continuous and not monotone.
	However, $\Klain:\Val_{n-1}^+\rightarrow C(\Grass_{n-1})$ is an order isomorphism.
\end{corollary}

\begin{proof}
	Assume that the inverse of the Klain map is continuous.
	Then there exists a constant $C\geq 0$ such that $\|\phi\|\leq C\|\Klain_\phi\|$ for each $\phi\in\Val_i^+$.
	This clearly implies that for every convex body $K$ there exists a constant $C_K\geq 0$ such that
		$$|\phi(K)|\leq C_K\|\Klain_\phi\|$$
	for every $\phi\in\Val_i^+$.
	Using Theorem~\ref{thmGi2}, we obtain $\calG(i)=\calG(n-1)$.
	If $i\neq n-1$, this contradicts Theorem~\ref{proppoly}.

	Next we show that if the inverse of the Klain map was monotone, then it would also be continuous.
	In fact, in this case $-\|\Klain_\phi\| \leq \Klain_\phi\leq \|\Klain_\phi\|$ implies $|\phi(K)|\leq V_i(K)\|\Klain_\phi\|$,
	which gives continuity.
	Hence we conclude that if  $i\neq n-1$, then the inverse of the Klain map $\Klain^{-1}: \img(\Klain) \rightarrow \Val_i^+$ cannot be monotone.

	Suppose now $i=n-1$.
	We define a map $T:C(\Grass_{n-1})\rightarrow \Val_{n-1}^+$ by
		$$Tf(K)=\frac{1}{2}\int_{S^{n-1}} f(u^\perp)\; dS_{n-1}(K,u).$$
	Since
		$$(\Klain\circ T)(f)=f$$
	we conclude that $\Klain:\Val_{n-1}^+\rightarrow C(\Grass_{n-1})$ is surjective and $T=\Klain^{-1}$.
	Since the area measure $S_{n-1}(K,\;\cdot\;)$ of a convex body $K$ is a positive measure,
	we conclude that $\Klain$ is an order isomorphism.
\end{proof}

Corollary~\ref{corinvklain} is in sharp contrast with the following result due to Alesker and Bernstein.
Here $(\Val_i^+)^\infty$ and $C^\infty(\Grass_i)$ are equipped with their natural $C^\infty$-topologies
and all topological concepts refer to these topologies.

\begin{theorem}[Alesker and Bernstein \cite{alesker_bernstein04}]
	The Klain map $\Klain: (\Val_i^+)^\infty\rightarrow C^\infty(\Grass_i)$, $0<i<n$, is an isomorphism of topological vector spaces onto its image.
	In particular, the image of the Klain map is closed.
\end{theorem}

We can use Corollary~\ref{corinvklain} to deduce that Alesker's Fourier transform $\FF:\Val^\infty \rightarrow \Val^\infty$
cannot be extended to a continuous map $\tilde\FF:\Val\rightarrow \Val$.
Indeed, assume that such an extension $\tilde\FF$ existed.
If $\phi\in(\Val_i^+)^\infty$, then it is known that $\FF\phi\in(\Val_{n-i}^+)^\infty$ and
\begin{equation}\label{deffourier}
	\Klain_{\FF \phi}(E)=\Klain_\phi(E^\perp),\qquad E\in\Grass_{n-i},
\end{equation}
see \cite{alesker03} (there the Fourier transform was denoted by $\DD$ and called duality transform).
In particular, we see that $\FF\circ \FF(\phi)=\phi$ for each $\phi\in(\Val^+)^\infty$.
By continuity, we obtain $\tilde\FF\circ\tilde\FF(\phi)=\phi$ for every $\phi\in\Val^+$,
which shows that $\tilde\FF:\Val^+\rightarrow\Val^+$ is a linear isomorphism of Banach spaces.
Extending \eqref{deffourier} by continuity, we arrive at the following commutative diagram:

\begin{center}
\begin{tikzpicture}[description/.style={fill=white,inner sep=2pt}]
    \matrix (m) [matrix of math nodes, row sep=3em,
    column sep=2.5em, text height=1.5ex, text depth=0.25ex]
    { \Val_{n-1}^+ &  \Val_{1}^+ \\
         C(\Grass_{n-1}) &  C(\Grass_{1})  \\ };
    \path[->,font=\scriptsize]
    (m-1-1) edge node[auto] {$ \tilde \FF$} (m-1-2)
    (m-1-1) edge node[auto] {$ \Klain $} (m-2-1)
    (m-2-1) edge node[auto] {$\perp$} (m-2-2)
	 (m-1-2) edge node[auto] {$\Klain$} (m-2-2);
\end{tikzpicture}
\end{center}

Since $\Klain: \Val_{n-1}^+\rightarrow C(\Grass_{n-1})$ is a linear isomorphism by Corollary~\ref{corinvklain},
we conclude that $\Klain: \Val_1^+\rightarrow C(\Grass_1)$ is a linear isomorphism as well.
This, however, contradicts Corollary~\ref{corinvklain}.

It was shown by Alesker \cite{alesker03} that the map $A_i:C^\infty(\Grass_i)\rightarrow (\Val_i^+)^\infty$ given by
	$$A_i(f)(K)=\int_{\Grass_i} \vol_i(K|E) f(E)\; dE$$
is surjective.
Our next result shows once more that the assumption of smoothness is crucial.
If we consider only continuous valuations, the corresponding statement is false,
even if we extend the domain of $A_i$ to $M(\Grass_i)$, the space of signed Borel measures on $\Grass_i$,
	$$A_i(\mu)(K)=\int_{\Grass_i} \vol_i(K|E) \; d\mu(E).$$
The Klain function of $A_i(\mu)\in\Val_i^+$ coincides with the cosine transform of $\mu$
\begin{equation}\label{eqklaincos}
	\Klain( A_i(\mu))=C_i(\mu),
\end{equation}
see e.g.\ \cite{alesker03}.

\begin{corollary}
	The map $A_i: M(\Grass_i)\rightarrow \Val_i^+$, $1<i<n$, is not surjective.
\end{corollary}

\begin{proof}
	By Theorem~\ref{proppoly} we can choose a centrally symmetric convex body $L\in \calK(\RR^n)$ which is not contained in $\calG(n-i)$.
	Define a valuation $\phi\in\Val_{i}^+$ by
		$$\phi(K)=V(K[i],L[n-i]).$$
	If $\phi=A_i(\mu)$ for some signed measure $\mu$, using \eqref{eqprojmixed}, \eqref{eqklaincos}, and \eqref{euclid_eqcosperp}, we obtain
		$$\vol_{n-i}(L|E)=\binom{n}{i} \Klain_\phi(E^\perp)=\binom{n}{i}(C_{n-i}\mu^\perp)(E),$$
	for each $E\in \Grass_{n-i}$. This contradicts our choice of $L$.
\end{proof}

Another consequence of Theorem~\ref{thmGi2} is related to the problem of describing the subspace of angular valuations. Recall that $\phi\in\Val_i^+$ is called \emph{angular} if for every polytope $P$
$$\phi(P)=\sum_{F\in\calF_i(P)} \Klain_\phi(\bar F) \gamma(F,P)\vol_i(F).$$
Here $\gamma(F,P)$ denotes the normalized exterior angle of $P$ at its face $F$
and $\bar F$ the unique translate of the affine hull of the face $F$ which contains the origin.
Important examples of angular valuations are the classical intrinsic volumes and the Hermitian intrinsic volumes.
Moreover, every constant coefficient valuation is known to be angular, see \cite{bernig_etal}*{Lemma 2.29}.
Examples of valuations which are not angular are harder to come by and not very much seems to be known concerning the class of all angular valuations.
As a consequence of Theorem~\ref{thmGi2} we obtain that there exists a non-trivial obstruction to the angularity of valuations.
In particular, we see that the smoothness class of $\phi$ is not important.
For a similar phenomenon see \cite{alesker12}.

\begin{corollary}
	Every $\phi\in\Val_{n-1}^+$ is angular.
	If $0<i<n-1$, then there exists a $\phi\in(\Val_i^+)^\infty$ which is not angular. 
\end{corollary}

\begin{proof}
Suppose $\phi\in\Val_{n-1}^+$. From the proof of Corollary~\ref{corinvklain} we immediately obtain 
\begin{align*}
\phi(P)	&=\frac{1}{2} \int_{S^{n-1}} \Klain_\phi(u^\perp)\; dS_{n-1}(P,u)\\
		&=\sum_{F\in\calF_{n-1}(P)} \Klain_\phi(\bar F) \gamma(F,P)\vol_{n-1}(F).
\end{align*}
Thus, every $\phi\in\Val_{n-1}^+$ is angular. 

Now fix $0<i<n-1$ and let $\phi\in(\Val_i^+)^\infty$. If $\phi$ is angular, then
\begin{align*}
|\phi(P)|	&\leq \| \Klain_\phi\| \sum_{F\in\calF_i(P)} \gamma(F,P)\vol_i(F)\\
			&= \| \Klain_\phi\| V_i(P).
\end{align*}
If we assume that every $\phi\in(\Val_i^+)^\infty$ is angular,
then the above inequality implies that \eqref{eqbounded} holds for every $\phi\in(\Val_i^+)^\infty$.
Since smooth valuations lie dense and the Klain map is continuous, \eqref{eqbounded} would in fact hold for every $\phi\in\Val_i^+$.
By Theorem~\ref{thmGi2}, we obtain that every centrally symmetric polytope is contained in $\calG(i)$.
This contradicts Theorem~\ref{proppoly}.
\end{proof}

Before we proceed, let us review a representation result for positive, linear functionals due to Choquet \cite{choquet69}*{Theorem 34.6}.
Let $X$ be a locally compact, Hausdorff space and $C(X,\RR)$ the vector space of continuous, real-valued functions on $X$ with the usual ordering:
$f\leq g$ if and only if $f(x)\leq g(x)$ for all $x\in X$.
For each subspace $H\subset C(X,\RR)$ we denote by $H^+$ those $f\in H$ with $f\geq0$.
Let $f,g \in C(X,\RR)^+$.
We say that $f$ dominates $g$ if for any $\varepsilon>0$ there is an $h\in C_c(X,\RR)$ such that $g\leq \varepsilon f+h$.
A subspace $H\subset C(X,\RR)$ is called \emph{adapted} if the following holds:
\begin{itemize}
	\item[(i)] $H=H^+- H^+$;
	\item[(ii)] for all $x\in X$ there is an $f\in H^+$ such that $f(x)>0$;
	\item[(iii)] every $g\in H^+$ is dominated by some $f\in H^+$.
\end{itemize}
Observe that if $X$ is compact, then (iii) is automatically satisfied.

\begin{theorem}[Choquet \cite{choquet69}*{Theorem 34.6}]\label{thmchoquet}
	Let $X$ be a locally compact Hausdorff space and $H\subset C(X,\RR)$ an adapted subspace of continuous functions.
	Suppose $\Lambda: H\rightarrow \RR$ is a positive, linear functional, i.e.\ $\Lambda (f)\geq 0$ whenever $f\in H^+$.  
	Then there exists a positive Radon measure $\mu$ such that every $f\in H$ is $\mu$-integrable and
		$$\Lambda(f)=\int_X f\; d\mu$$
	whenever $f\in H$.
\end{theorem}

\begin{theorem}\label{thmKi}
	Let $K\in\calK(\RR^n)$ be centrally symmetric and $0<i<n$.
	Then $K\in\calK(i)$ if and only if
	\begin{equation}\label{eqpos}
		0\leq\Klain_\phi \ \Longrightarrow\ 0\leq \phi(K),
	\end{equation}
	for every $\phi\in\Val_i^+$.
\end{theorem}

\begin{proof}
	Suppose $K\in\calK(i)$.
	Since $\calK(i)\subset\calG(i)$, we obtain as in the first part of the proof of Theorem~\ref{thmGi2} that
		$$\phi(K)=\int_{\Grass_{i}} \Klain_\phi(E)\; d\mu_K(E),$$
	but now with some positive Borel measure $\mu_K$.
	This immediately yields (\ref{eqpos}).

	Conversely, assume that (\ref{eqpos}) holds. Note that the subspace $\img(\Klain) \subset C(\Grass_i)$ is adapted.
	In fact, we have $\Klain_\phi=\|\Klain_\phi\|-(\|\Klain_\phi\|-\Klain_\phi)$ for any $\phi\in\Val_i^+$; the other conditions are trivial.
	Hence the positive linear functional $f\mapsto \Klain^{-1}(f)(K)$ defined on $\img(\Klain)$
	can be represented by some positive Radon measure $\mu_K$ on $\Grass_i$.
	Thus, $\phi(K)=\int_{\Grass_{i}} \Klain_\phi(E)\; d\mu_K(E)$ for any $\phi\in\Val_i^+$ and as in the proof of Theorem~\ref{thmGi2}
	we conclude that $K\in\calK(i)$.
\end{proof}

It is well-known that $\calK(1)$, the class of zonoids, is a closed subset of $\calK(\RR^n)$, see e.g.\ \cite{schneider93}.
The same holds true for all the other classes $\calK(i)$, $0<i<n$.
The second named author was informed by W.~Weil that this result can be deduced from \cite{goodey_weil91}*{Corollary 5.2}.
We prefer to give a new proof using Theorem~\ref{thmKi}.

\begin{corollary}
	For $0<i<n$ the class $\calK(i)$ is a closed subset of $\calK(\RR^n)$.
\end{corollary}

\begin{proof}
	Let $K_m\in\calK(i)$, $m\geq0 $, be a sequence of convex bodies converging to a convex body $K$.
	If $\phi\in\Val^+_i$ satisfies $\Klain_\phi\geq 0$, then $\phi(K)=\lim_{m\rightarrow\infty} \phi(K_m)\geq0$, by Theorem~\ref{thmKi}.
	Consequently, $K\in\calK(i)$ and we conclude that $\calK(i)$ is closed.
\end{proof}

	\section{Positive valuations and McMullen's decomposition theorem}
		Let $\rmR(K)$ denote the circumradius and $\rmr(K)$ the inradius of a convex body $K\in \calK(\RR^n)$.
If $K$ has empty interior, then by definition $\rmr(K)$ is computed inside $\aff K$.
The successive outer and inner radii of a convex body $K\in\calK(\RR^n)$ are the nonnegative numbers defined by
	$$\rmR_i(K)=\min_{E\in\Grass_i} \rmR(K|E)\qquad\quad \rmr_i(K)=\max_{F\in \overline\Grass_i} \rmr(K\cap F)$$
for $i=1,\ldots,n$.
Observe that $\rmR_n(K)=\rmR(K)$, $\rmr_n(K)=\rmr(K)$, and
	$$\rmR_1(K)\leq \rmR_2(K)\leq \cdots \leq \rmR_n(K)\quad \text{and}\quad \rmr_1(K)\geq \rmr_2(K)\geq\ldots \geq \rmr_n(K)$$
for any $K\in\calK(\RR^n)$.
It was shown by Perel\cprime man \cite{perelman87} that
\begin{equation}\label{eqperelman}
	\frac{\rmR_{n-i+1}(K)}{\rmr_i(K)}\leq i+1.
\end{equation}
For more information on successive radii see e.g.\ \cites{gonzalez_hernandez,henk_hernandes08} and the references there.

\begin{lemma}\label{lempos}
	Let $\phi\in \Val_1^+$ and suppose that $\Klain_\phi >0$.
	Then there exist constants $c_0,c_1\in\RR$ such that $c_0+\phi(K)+c_1 V_2(K)\geq 0$ for every $K\in\calK(\RR^n)$.
\end{lemma}

\begin{proof}
	If $n=2$, then the lemma follows immediately from Corollary~\ref{corinvklain}.  We assume therefore $n\geq 3$.
	Let $B^n\subset \RR^n$ denote the Euclidean unit ball and put
		$$\calM=\{K\in\calK(\RR^n): K\subset B^n,\ 2\rmR(K)\geq 1\}$$
	and
		$$\calN=\{K\in\calM: \dim K=1\}.$$
	Both $\calM$ and $\calN$ are compact. Set $\varepsilon=\min_{K\in\calN} \phi(K)>0$.
	Since $\phi$ is continuous, $\phi$ is uniformly continuous on compact sets.
	Hence there exists $0<\eta< \frac{1}{12}$ such that $d_H(K,L)< 6\eta$ implies $|\phi(K)-\phi(L)|<\varepsilon$ for all $K,L\in\calM$.
	Put
		$$c_0=\|\phi\|\quad \text{and}\quad c_1=\frac{\|\phi\|}{\pi \eta^2}.$$

	Let $K\in\calK(\RR^n)$ and suppose that $\rmR(K)>1$ and $\frac{\rmr_2(K)}{\rmR(K)}<\eta$.
	Put $K'=\frac{1}{\rmR(K)} K+x$ for a suitable $x\in\RR^n$ such that $K'\in \calM$.
	Clearly, $\rmR(K')=1$ and $\rmr_2(K')<\eta$.
	Using \eqref{eqperelman} with $i=2$, we deduce that $\rmR_{n-1}(K')< 3\eta$.
	By definition, $R_{n-1}(K')<3\eta$ implies that $K'$ is contained in a cylinder of radius less than $3\eta$.
	Let the axis of this cylinder be parallel to, say, $u\in S^{n-1}$.
	Choose $p_1,p_2\in K'$ such that $u\cdot p_1=h(K',u)$ and $-u\cdot p_2=h(K',-u)$ and put $L'= [p_1,p_2]$,
	where $[p_1,p_2]$ denotes the line segment between $p_1$ and $p_2$.
	Thus, there exists a convex body $L'\subset B^n$ with $\dim L'=1 $ such that $d_H(K',L')<6\eta$.
	Since
		$$\rmR(L')\geq \rmR(K')-d_H(K',L')> 1-6\eta>\frac{1}{2},$$
	we conclude that $L'\in\calN$.
	Therefore we have
		$$\phi(K)=\rmR(K)\phi(K')> \rmR(K)(\phi(L')-\varepsilon)\geq 0.$$

	Now suppose that $\rmR(K)>1$ and $\frac{\rmr_2(K)}{\rmR(K)}\geq\eta$.
	From the monotonicity of the intrinsic volumes and the trivial estimate $|\phi(K)|\leq \|\phi\|\rmR(K)$, we obtain
		$$V_2(K)\geq \pi \rmr_2(K)^2\geq \pi\eta^2\rmR(K)^2\geq \pi\eta^2\rmR(K)\geq \frac{\pi\eta^2}{\|\phi\|} |\phi(K)|.$$

	Finally let $K\in\calK(\RR^n)$ be such that $\rmR(K)\leq 1$.
	In this case we obviously have
		$$|\phi(K)|\leq \|\phi\|.$$
	In any case we see that $c_0+\phi(K)+c_1 V_2(K)\geq 0$, which proves the lemma.
\end{proof}

The following theorem shows that an analog of McMullen's decomposition theorem
does not hold true in the class of positive, translation-invariant, continuous valuations.

\begin{theorem}\label{thmposmc}
	If $n\geq3$, then there exists a positive, even, translation-invariant, continuous valuation on $\calK(\RR^n)$
	such that not all of its homogeneous components are positive.
\end{theorem}

\begin{proof}
	Since the inverse of the Klain map $\Klain:\Val_1^+\rightarrow C(\Grass_1)$ is not monotone by Corollary~\ref{corinvklain},
	there exists $\phi\in\Val_1^+$ such that $\Klain_\phi\geq 0$ but $\phi(K)<0$ for some $K\in\calK(\RR^n)$.
	If necessary replacing $\phi$ by $\phi+tV_1$ with $t>0$ small, we may assume that $\Klain_\phi> 0$ and $\phi(K)<0$ for some $K\in\calK(\RR^n)$.
	By Lemma \ref{lempos} there exist constants $c_0,c_1\in\RR$ such that
		$$c_0+\phi+c_1 V_2$$
	is a positive valuation.
	This gives the desired counterexample.
\end{proof}

As an immediate consequence of Theorem~\ref{thmposmc} we obtain that there exists no McMullen decomposition for Minkowski valuations if $n\geq 3$.
\begin{corollary}
	If $n\geq3$, then there exists an even, translation-invariant, continuous Minkowski valuation $\Phi:\calK(\RR^n)\rightarrow\calK(\RR^n)$
	which cannot be decomposed into a sum of homogeneous Minkowski valuations.
\end{corollary}

\begin{proof}
	Let $\phi$ be a positive, even, translation-invariant, continuous valuation on $\calK(\RR^n)$
	such that one of its homogeneous components $\phi_i$, say $\phi_{i_0}$, is not a positive valuation.
	Such a valuation exists by Theorem~\ref{thmposmc}.
	Define an even, translation-invariant, continuous Minkowski valuation by
		$$ \Phi K = \phi(K) B^n $$
	for all $K \in \calK(\RR^n)$.
	Assume that $\Phi$ can be written as a sum of homogeneous components, $\Phi=\Phi_0+\cdots+\Phi_n$.
	This implies
		$$\sum_{i=0}^nh(\Phi_i K,x)=h(\Phi K,x)=\phi(K)|x|=\sum_{i=0}^n \phi_i(K)|x|,$$
	and hence by the uniqueness of McMullen's decomposition, we obtain in particular $h(\Phi_{i_0} K,x)=\phi_{i_0}(K)|x|$.
	This is a contradiction, since if $\phi_{i_0}(K)<0$, then $\phi_{i_0}(K)|x|$ is clearly no support function.
\end{proof}

	\section*{Acknowledgment}
		The first named author was supported by the  Austrian Science Fund (FWF), within
		the project ``Linearly intertwining maps on convex bodies'', Project number: P\,23639-N18.

		The second named author was supported by the  Austrian Science Fund (FWF), within
		the project ``Minkowski valuations and geometric inequalities'', Project number: P\,22388-N13.

\end{document}